\documentclass[12pt,a4paper]{article}
\usepackage{etex}
\usepackage{latexsym,amsthm,amsmath,amssymb}

\newtheorem{theorem}{Theorem}[section]

\newtheorem{corollary}[theorem]{Corollary}

\newtheorem{remark}[theorem]{Remark}

\title{Characterization of Completely $k$-Magic Regular Graphs}
\author{Arnold A. Eniego \\ Science and Mathematics Department \\ National University \\ Manila, The Philippines \\ aaeniego@national-u.edu.ph \and I.J.L. Garces \\ Department of Mathematics \\ Ateneo de Manila University \\ Quezon City, The Philippines \\ ijlgarces@ateneo.edu}
\date{25 March 2016}

\begin{document}

\maketitle

\begin{abstract}
Let $k \in \mathbb{N}$ and $c \in \mathbb{Z}_k$, where $\mathbb{Z}_1=\mathbb{Z}$. A graph $G=(V(G),E(G))$ is said to be $c$-sum $k$-magic if there is a labeling $\ell:E(G) \rightarrow \mathbb{Z}_k \setminus \{0\}$ such that $\sum_{u \in N(v)} \ell(uv) \equiv c \pmod{k}$ for every vertex $v$ of $G$, where $N(v)$ is the neighborhood of $v$ in $G$. We say that $G$ is completely $k$-magic whenever it is $c$-sum $k$-magic for every $c \in \mathbb{Z}_k$. In this paper, we characterize all completely $k$-magic regular graphs.


\end{abstract}

\section{Introduction}
Let $G=(V(G),E(G))$ be a finite, simple (unless otherwise stated) graph with vertex set $V(G)$ and edge set $E(G)$. A \textit{factor} of $G$ is a subgraph $H$ with $V(H) = V(G)$. In particular, if a factor $H$ of $G$ is $h$-regular, then we say that $H$ is an \textit{$h$-factor} of $G$. An \textit{$h$-factorization} of $G$ is a partition of $E(G)$ into disjoint $h$-factors. If such factorization of $G$ exists, then we say that $G$ is \textit{$h$-factorable}.

The following theorem is attributed to Petersen \cite{Petersen}, which we state using the versions of Akiyama and Kano \cite{Akiyama&Kano} and Wang and Hu \cite{Wang&Hu}.

\begin{theorem}[{\cite[Theorem 3.1]{Akiyama&Kano}, \cite{Petersen}, \cite[Theorem 10]{Wang&Hu}}]\label{Petersen}
Let $G$ be a $2r$-regular connected general graph (not necessarily simple), where $r \geq 1$. Then $G$ is $2$-factorable, and it has a $2k$-factor for every $k$, $1 \leq k \leq r$. Moreover, if $G$ is of even order, then it is $r$-factorable.
\end{theorem}

A graph $G$ is \textit{$\lambda$-edge connected} if it remains connected whenever fewer than $\lambda$ edges are removed.

\begin{theorem}\label{regularfactors} {\rm \cite{Gallai}} Let $r$ and $k$ be integers such that $1 \leq k < r$, and $G$ be a $\lambda$-edge connected $r$-regular general graph, where $\lambda \geq 1$. If one of the following conditions holds:
	\begin{enumerate}
		\item[{\rm (1)}] $r$ is even, $k$ is odd, $|G|$ is even, and $\frac{r}{\lambda} \leq k \leq r(1-\frac{1}{\lambda})$,
		\item[{\rm (2)}] $r$ is odd, $k$ is even, and $2 \leq k \leq r(1-\frac{1}{\lambda})$, or
		\item[{\rm (3)}] $r$ and $k$ are both odd and $\frac{r}{\lambda} \leq k$,
	\end{enumerate}
then $G$ has a $k$-regular factor.
\end{theorem}

Let $A$ be a non-trivial Abelian group written additively. A finite simple graph $G=(V(G),E(G))$ is said to be \textit{$A$-magic} if there exists an edge labeling $\ell:E(G) \rightarrow A \setminus \{0\}$ such that the induced vertex labeling $\ell^+:V(G) \rightarrow A$, defined by $\ell^+(v) = \sum_{uv\in E(G)} \ell(uv)$, is a constant map. If $c \in A$ and $\ell^+(v) = c$ for all $v \in V(G)$, then we call $c$ is a \textit{magic sum} of $G$.

Let $k$ be a positive integer. If $G$ is $\mathbb{Z}_k$-magic graph, then we say that $G$ is \textit{$k$-magic}. Here, $\mathbb{Z}_1=\mathbb{Z}$ the group of integers, and $\mathbb{Z}_k=\{0,1,2,\ldots,k-1\}$ the group of integers modulo $k \geq 2$. In particular, if $G$ is $k$-magic with magic sum $c$, then we say that $G$ is \textit{$c$-sum $k$-magic}. If $G$ is $c$-sum $k$-magic for all $c \in \mathbb{Z}_k$, then it is said to be \textit{completely $k$-magic}. The set of all magic sums $c \in \mathbb{Z}_k$ of $G$ is the \textit{sum spectrum of $G$ with respect to $k$} and is denoted by $\Sigma_k(G)$. If $c=0$, then we say that $G$ is \textit{zero-sum $k$-magic}. The \textit{null set} of $G$, denoted by $N(G)$, is the set of all positive integers $k$ such that $G$ is a zero-sum $k$-magic graph.

\begin{remark}\label{l'}
If $c \in \mathbb{Z}_k$ and $\ell$ is a $c$-sum $k$-magic labeling of $G$, then the labeling $\ell'$, defined by $\ell'(e) = k - \ell(e)$, is a $(k-c)$-sum $k$-magic labeling of $G$.
\end{remark}

\begin{remark}\label{no_completely-2-magic}
Any $2$-magic graph is not completely $2$-magic.
\end{remark}

The concept of $A$-magic graphs is due to Sedlacek \cite{Sedlacek}. Over the years, many papers have been published in connection with magic graphs. To name a few, Akbari, Rahmati, and Zare \cite{Akbari&Rahmati&Zare} and Choi, Georges, and Mauro \cite{Choi&Georges&Mauro} investigated the zero-sum $k$-magic labelings and null sets of regular graphs. Dong and Wang \cite{Dong&Wang} solved affirmatively a conjecture posed in \cite{Akbari&Rahmati&Zare} on the existence of a zero-sum $3$-magic labeling of $5$-regular graphs. Salehi \cite{Salehi1} determined the integer-magic spectra of certain classes of cycle-related graphs. Shiu and Low \cite{Shiu&Low2} analyzed the group-magic property for complete $n$-partite graphs and
composition graphs with deleted edges.

Using the term ``index set,'' Wang and Hu \cite{Wang&Hu} initially studied the concept of completely $k$-magic graphs. They gave a partial list of completely $1$-magic regular graphs. Eniego and Garces \cite{Eniego&Garces} completely added the remaining cases in this list. They also presented the sum spectra of some regular graphs that are not completely $k$-magic.

\begin{theorem}[{\cite[Theorem 13]{Akbari&Rahmati&Zare}}]\label{Akbari&Rahmati&ZareTheorem13}
Let $G$ be an $r$-regular graph, where $r \geq 3$ and $r \ne 5$. If $r$ is even, then $N(G)=\mathbb{N}$; otherwise, $\mathbb{N} \setminus \{2,4\} \subseteq N(G)$.
\end{theorem}

\begin{theorem}[{\cite[Theorem 2.1]{Dong&Wang}}]\label{Dong&Wang}
Every $5$-regular graph admits a zero-sum $3$-magic labeling.
\end{theorem}

\begin{theorem}[{\cite[Theorem 3.3]{Eniego&Garces}}]\label{c_n}
Let $n \geq 3$ and $k \ge 3$ be integers, and $C_n$ the cycle with $n$ vertices.
\begin{enumerate}
\item[{\rm (1)}] If $n$ is even, then $C_n$ is completely $k$-magic for all $k$.
\item[{\rm (2)}] If $n$ is odd, then $C_n$ is not completely $k$-magic for any $k$. Moreover, we have
    $$\Sigma_k(C_n) =
    \begin{cases}
    \mathbb{Z}_k \setminus \{0\} & \text{if $k$ is odd,} \\
    \{0, 2, \ldots, k-2 \} & \text{if $k$ is even}.
    \end{cases}$$
\end{enumerate}
\end{theorem}

\begin{theorem}[{\cite[Lemma 3.4]{Eniego&Garces}}]\label{odd order lemma}
Let $k \ge 4$ be an even integer. Then there exists no $k$-magic graph of odd order that is completely $k$-magic. In particular, if $c$ is a magic sum of a $k$-magic graph of odd order, then $c$ must be even.
\end{theorem}

\begin{theorem}[{\cite[Theorem 3.6]{Eniego&Garces}}]\label{TheoremAfterLemmaFromDummit&Foote}
Let $k,r \geq 3$ be integers, and $G$ an $r$-regular graph. If $\gcd(r,k) = 1$, then $\{1, 2, \ldots, k-1\} \subseteq \Sigma_k(G)$.
\end{theorem}

\begin{theorem}[{\cite[Theorem 3.7]{Eniego&Garces}}]\label{Magic&PerfectMatching}
	Let $G$ be a zero-sum $k$-magic $r$-regular graph, where $k \ge 3$ and $r\geq 3$. If $G$ has a $1$-factor, then $G$ is completely $k$-magic.
\end{theorem}

\begin{theorem}[{\cite[Theorem 13]{Wang&Hu}, \cite[Theorem 2.1]{Eniego&Garces}}]\label{Wang&Hu}
Let $G$ be an $r$-regular graph of order $n$. Then
$$\Sigma_1(G) =
\begin{cases}
\mathbb{Z}\setminus\{0\} & \text{if $r = 1$,}    \\
\mathbb{Z} & \text{if $r = 2$ and $G$ contains even cycles only,}    \\
2\mathbb{Z}\setminus\{0\} & \text{if $r = 2$ and $G$ contains an odd cycle,}    \\
2\mathbb{Z} & \text{if $r \geq 3$, $r$ even, and $n$  odd,}    \\
\mathbb{Z} & \text{if $r \geq 3$ and $n$ even,}
\end{cases}$$
where $2\mathbb{Z}$ is the set of all even integers.
\end{theorem}

With Remark \ref{no_completely-2-magic} and Theorem \ref{Wang&Hu}, it remains to characterize all completely $k$-magic regular graphs for $k\geq 3$. This characterization is the main theorem of this paper, which we state as follows.

\begin{theorem}[{\bf Main Theorem}]
Let $r \geq 2$ and $k \geq 3$ be integers, and $G$ an $r$-regular graph of order $n \geq 3$. Then $G$ is completely $k$-magic if and only if one of the following properties holds:
\begin{enumerate}
\item[{\rm (1)}] $k \geq 3$, $r=2$, and $G$ contains even cycles only,
\item[{\rm (2)}] $k \geq 5$ and $r \geq 3$ odd,
\item[{\rm (3)}] $k \geq 5$, $r \geq 4$ even, and $n$ even,
\item[{\rm (4)}] $k \geq 5$ odd, $r \geq 4$ even, and $n$ odd,
\item[{\rm (5)}] $k=4$, $r \geq 3$, $n$ even, and $G$ zero-sum $4$-magic, or
\item[{\rm (6)}] $k = 3$ and any one of the following conditions holds:
\begin{enumerate}
\item[{\rm (i)}] $r \not\equiv 0 \pmod{3}$,
\item[{\rm (ii)}] $r \equiv 0 \pmod{6}$, or
\item[{\rm (iii)}] $r \equiv 0 \pmod{3}$, $r$ odd, and $G$ has a factor $H$ such that $d_H(v) \equiv 1 \pmod{3}$ for all $v \in V(H)$.
\end{enumerate}
\end{enumerate}
\end{theorem}

For convenience, we assume that all graphs to be considered are finite and simple (unless otherwise stated). We also write $\mathbb{Z}_k^*$ to mean $\mathbb{Z}_k \setminus \{0\}$. For terms that are not defined in this paper, see \cite{Bondy&Murty}.

\section{Proof of the Main Theorem}

We divide the proof into several results.

It is not difficult to see that if $G$ is $1$-regular, then $\Sigma_k(G) = \mathbb{Z}_k^*$. For $2$-regular graphs, the following remark is a consequence of Theorem \ref{c_n}.

\begin{remark}\label{rc_n}
Let $k \geq 3$ and $G$ a $2$-regular graph. If $G$ has an odd cycle, then
$$\Sigma_k(G) =
\begin{cases}
\mathbb{Z}_k^* & \text{if $k$ is odd} \\
\{0, 2, \ldots, k-2 \} & \text{if $k$ is even.}
\end{cases}$$
Otherwise, we have $\Sigma_k(G) = \mathbb{Z}_k$.
\end{remark}

Clearly, if $G$ is $1$-factorable, then $G$ is completely $k$-magic. The following theorem considers regular graphs that has a factor that is completely $k$- magic.

\begin{theorem}\label{factor2}
Let $r \geq 2$, $2 \le h \le r$, $k \ne 2$, and $G$ an $r$-regular graph. If $G$ has an $h$-factor that is completely $k$-magic, then $G$ is completely $k$-magic.
\end{theorem}
\begin{proof}
The case when $h = r$ is trivial, so we assume $h < r$.  Let $H$ be an $h$-factor of $G$ that is completely $k$-magic.  Let $\alpha = c - (r-h) \pmod{k}$ and $f_{\alpha}$ be an $\alpha$-sum $k$-magic labeling of $H$ for each $c \in \mathbb{Z}_k$.
		
Define $\ell_c: E(G) \rightarrow \mathbb{Z}_k^*$ by
$$\ell_c(e) =
\begin{cases}
f_{\alpha}(e) & \text{if $e \in E(H)$ }    \\
1 & \text{if $e \in E(G \setminus H)$. }
\end{cases}$$
The sum of the labels of the edges incident to each vertex of $G$ is $c-(r-h) + (r-h) \equiv c \pmod{k}$. Thus, $\ell_c$ is a $c$-sum $k$-magic labeling of $G$ for each $c \in \mathbb{Z}_k$. Hence, $G$ is completely $k$-magic. 	
	\end{proof}





The following construction will be useful in the proofs of our succeeding results.

\begin{remark}\label{G'}
Let $G$ be an $r$-regular graph with $E(G) = \{e_1,e_2,e_3,\ldots,e_m\}$, where $r \geq 1$. Then we can construct a graph $G'$ (with parallel edges) such that $V(G')=V(G)$ and $E(G')=E(G) \cup \{e_1',e_2',e_3',\ldots,e_m'\}$, where $e_i'$ is a duplicate edge of $e_i$ in $G$ for each $i$ (that is, edges $e_i$ and $e_i'$ have the same end vertices). By Theorem \ref{Petersen}, $G'$ has a $2h$-factor $H'$ for each $h$, $1 \leq h \leq r$. Also, $G' \setminus H'$ is a $(2r-2h)$-factor of $G'$ obtained by removing the edges of $H'$ from $G'$.
\end{remark}

\begin{theorem}\label{5-regular}
Let $G$ be a $5$-regular graph. Then $\mathbb{N} \setminus \{2,4\} \subseteq N(G)$.
\end{theorem}

\begin{proof}
We know from Theorem \ref{Wang&Hu} and Theorem \ref{Dong&Wang} that $1,3 \in N(G)$. For $k \geq 5$, we consider two cases.

\textsc{Case 1.} Suppose $k \geq 5$ and $k \neq 8$. Using the construction and notation described in Remark \ref{G'}, let $H'$ be a $2$-factor of $G'$. Then $G' \setminus H'$ is an $8$-factor of $G'$.
	
Define a zero-sum $k$-magic labeling $\ell'$ on $G'$ by
$$\ell'(e) =
\begin{cases}
k-4 & \text{if $e \in E(H')$ } \\
1 & \text{if $e \in E(G' \setminus H')$.}
\end{cases}$$
Note that the labeling $\ell$ on $G$ defined by $\ell(e_i) = \ell'(e_i) + \ell'(e_i')$ for $e_i \in E(G)$ is a zero-sum $k$-magic labeling on $G$.

\textsc{Case 2.} Suppose $k=8$. Using again the construction in Remark \ref{G'}, let $H'$ be a $4$-factor of $G'$. Then $G' \setminus H'$ is a $6$-factor of $G'$.

Define a zero-sum labeling $\ell'$ on $G'$ by
$$\ell'(e) =
\begin{cases}
2 & \text{if $e \in E(H')$ } \\
4 & \text{if $e \in E(G' \setminus H')$.}
\end{cases}$$
Observe that the labeling $\ell$ on $G$ defined by $\ell(e_i) = \frac{1}{2}[\ell'(e_i)+\ell'(e_i')]$ for $e_i \in E(G)$ is a zero-sum $8$-magic labeling on $G$.
	
Therefore, $\mathbb{N} \setminus \{2, 4\} \subseteq N(G)$.
\end{proof}

Note that an odd-regular graph may not be zero-sum $4$-magic. It was remarked in \cite[Remark 10]{Akbari&Rahmati&Zare} that an odd-regular graph $G$ is not zero-sum $4$-magic if $G$ has a vertex such that every edge incident to it is a cut-edge.

\begin{theorem}\label{odd-regular}
Let $G$ be an $r$-regular graph, where $r \geq 3$ is odd and $k \geq 5$. Then $G$ is completely $k$-magic.
\end{theorem}

\begin{proof}
We know from Theorems \ref{Akbari&Rahmati&ZareTheorem13} and \ref{5-regular} that $0 \in \Sigma_k(G)$. Let $E(G)=\{e_1,e_2,e_3,\ldots,e_m\}$. As constructed in Remark \ref{G'}, let $H'$ and $G' \setminus H'$ be a $2$-factor and $(2r-2)$-factor of $G'$, respectively. We consider two cases.
	
\textsc{Case 1.} Suppose $r \equiv 1 \pmod{k}$. Then $\gcd(r,k) = 1$. By Theorem \ref{TheoremAfterLemmaFromDummit&Foote}, $G$ is completely $k$-magic.
	
\textsc{Case 2.} Suppose $r \not\equiv 1 \pmod{k}$. Assume $\gcd(r,k) = d$ so that $r = ad$ and $k = bd$ for some positive integers $a$ and $b$. Note that, since $r$ is odd, $d$ is also odd. We consider two sub-cases.
	
\textsc{Sub-Case 2.1.} Suppose $k \geq 5$ is odd. Then $b$ is odd.
	
For each $c \in \mathbb{Z}_k^* \setminus \{k-b, k-2b\}$, define $\ell_c':E(G') \to \mathbb{Z}_k^*$ by
$$\ell_c'(e) =
\begin{cases}
x & \text{if $e \in E(H')$} \\
\frac{1}{2}(k+b) & \text{if $e \in E(G' \setminus H')$,}
\end{cases} $$
where $x = \frac{1}{2}(b+c)$ if $c$ is odd, and $x = \frac{1}{2}(b+c+k)$ if $c$ is even. Observe that $\ell_c'$ is a $c$-sum $k$-magic labeling of $G'$ for each $c \neq 0$.
	
For each $c \not\in \{0,k-b,k-2b\}$, define $\ell_c:E(G) \to \mathbb{Z}_k ^*$ by $\ell_c(e_i) = \ell_c'(e_i) + \ell_c'(e_i')$ for $1 \leq i \leq m$. Since $\ell_c'$ is a $c$-sum $k$-magic labeling of $G'$, $\ell_c$ is a $c$-sum $k$-magic labeling of $G$ for each $c \in \mathbb{Z}_k^* \setminus \{k-b,k-2b\}$.

	
If $k \neq 3b$, then, by Remark \ref{l'}, $k-b, k-2b \in \Sigma_k(G)$. If $k = 3b$, it is enough to show that $k-2b \in \Sigma_k(G)$. To do that, we provide a different labeling using a different set of factors of $G'$. Let $J'$ and $G'\setminus J'$ be a $4$-factor and $(2r-4)$-factor of $G'$ respectively. In addition, we let $J' = J_1' \cup J_2'$, where $J_1'$ and $J_2'$ are $2$-factors of $J'$.

Define $\ell{}':E(G') \to \mathbb{Z}_k^*$ by
$$\ell'(e) =
\begin{cases}
\frac{b+1}{2} & \text{if $e \in E(J_1')$} \\
\frac{b-1}{2} & \text{if $e \in E(J_2')$} \\
b & \text{if $e \in E(G' \setminus J').$}
\end{cases} $$

Since $k = 3b$, $d = 3$ and $r = 3a$. Thus, the magic sum in $G'$ is given by $2(\frac{b+1}{2}) + 2(\frac{b-1}{2}) + b(2r-4) \equiv -2b \pmod{k}$. Define $\ell_{}:E(G) \to \mathbb{Z}_k ^*$ by $\ell_{}(e_i) = \ell'(e_i) + \ell'(e_i')$ for $1 \leq i \leq m$. Note that $\ell$ is also a $(k-2b$)-sum $k$-magic labeling of $G$.
	
\textsc{Sub-Case 2.2.} Suppose $k \geq 6$ is even. Then $b$ is even.
	
By labeling all the edges of $G$ with $\frac{1}{2}k$, we see that $\frac{1}{2}k \in \Sigma_k(G)$.
	
Suppose $r-1 \equiv \frac{1}{2}k \pmod{k}$. For each $c \in \mathbb{Z}_k^* \setminus \{k-1, \frac{1}{2}k\}$, define $\ell_c':E(G') \rightarrow \mathbb{Z}_k ^*$ by
$$\ell_c'(e) =
\begin{cases}
c & \text{if $e \in E(H')$ } \\
1 & \text{if $e \in E(G' \setminus H')$. }
\end{cases} $$
Observe that the sum of the labels of the edges incident to each vertex in $G'$ is $ 2(r-1) + 2c \equiv 2c \pmod{k}$. Using a similar argument as in Sub-Case 2.1, it can be shown that $G$ is also $e$-sum $k$-magic for all even $e \neq 0$. Thus, we are left to show that $G$ is $c$-sum $k$-magic as well for all odd $c$.
	
For each odd $c \neq k-1$, define $\ell_c:E(G) \rightarrow \mathbb{Z}_k ^*$ by $\ell_c(e_i) = \frac{1}{2}[\ell_c'(e_i) + \ell_c'(e_i')]$ for each $i$, $1 \leq i \leq m$. Note that, since $\ell_c'$ is a $2c$-sum $k$-magic labeling of $G'$, $\ell_c$ is a $c$-sum $k$-magic labeling of $G$ for each odd  $c\neq k-1$. Again, by Remark \ref{l'}, we see that $k-1 \in \Sigma_k(G)$.
	
Suppose $r-1 \equiv r_0 \pmod{k}$, where $r_0 \neq \frac{1}{2}k$. For each $c \in \mathbb{Z}_k^* \setminus \{r_0, r_0 + \frac{1}{2}k, r_0-1\}$, define $\ell_c':E(G') \to \mathbb{Z}_k^*$ by
$$\ell_c'(e) =
\begin{cases}
c-r_0 & \text{if $e \in E(H')$} \\
1 & \text{if $e \in E(G' \setminus H')$.}
\end{cases} $$
Observe that the sum of the labels of the edges incident to each vertex in $G'$ is $2r_0 + 2c - 2r_0 \equiv 2c \pmod{k}$. As in Sub-Case 2.1, it can be shown that $G$ is also even-sum $k$-magic. So again, we are left to show that $G$ is odd-sum $k$-magic.
	
As what we did earlier, for each odd $c \neq r_0-1$ (and, possibly, $r_0 + \frac{1}{2}k$), define $\ell_c:E(G) \to \mathbb{Z}_k ^*$ by $\ell_c(e_i) = \frac{1}{2}[\ell_c'(e_i) + \ell_c'(e_i')]$ for all $i$, $1 \leq i \leq m$. Since $\ell_c'$ is a $2c$-sum $k$-magic labeling of $G'$, $\ell_c$ is a $c$-sum $k$-magic labeling of $G$ for each odd $c \neq r_0-1$ (and, possibly, $r_0 + \frac{1}{2}k)$. If $r_0 - 1$ and $r_0 + \frac{1}{2}k$ are not inverses, then, by Remark \ref{l'}, $\mathbb{Z}_k^* \subset \Sigma_k(G)$.
	
If $r_0 - 1$ and $r_0 + \frac{1}{2}k$ are inverses, then it is enough to show that $r_0 - 1 \in \Sigma_k(G)$. Define $\ell'$ on $G'$ by
$$\ell'(e) =
\begin{cases}
k-1 & \text{if $e \in E(H')$ } \\
1 & \text{if $e \in E(G' \setminus H')$.}
\end{cases}$$
Note that the magic sum using $\ell'$ is $2r_0 - 2$. Define $\ell$ on $G$ by $\ell(e_i) = \frac{1}{2}[\ell'(e_i) + \ell'(e_i')]$ for $e_i \in E(G)$. Clearly, $\ell$ is an $(r_0 - 1)$-sum $k$-magic labeling on $G$. Thus, by Remark \ref{l'}, $r_0 + \frac{1}{2}k \in \Sigma_k(G)$, and so $\mathbb{Z}_k^* \subset \Sigma_k(G)$.
	
In any case, $G$ is completely $k$-magic.
\end{proof}

\begin{theorem}\label{2r-regular}
Let $k \geq 5$ and $G$ a $2r$-regular graph of order $n \geq 3$, where $r \geq 2$.
\begin{enumerate}
\item[{\rm (1)}] If $n$ is even, then $G$ is completely $k$-magic.
\item[{\rm (2)}] If $n$ is odd, then
\begin{enumerate}
\item[{\rm (i)}] $G$ is completely $k$-magic if $k$ is odd, and
\item[{\rm (ii)}] $\Sigma_k(G) = \{0, 2, 4, \dots , k-2\}$ if $k$ is even.
\end{enumerate}
\end{enumerate}
\end{theorem}

\begin{proof}
Let $E(G)=\{e_1,e_2,e_3,\ldots,e_m\}$. By Theorem \ref{Akbari&Rahmati&ZareTheorem13}, $G$ is zero-sum $k$-magic.
	
(1) Suppose $r=2$. To prove the theorem, we only show that $\mathbb{Z}_k^* \subset \Sigma_k(G)$. We consider two cases.

\textsc{Case 1.} Suppose $k$ is odd. Then $\gcd(4,k) = 1$. By Theorem \ref{TheoremAfterLemmaFromDummit&Foote}, $\mathbb{Z}_k^* \subseteq \Sigma_k(G)$.

\textsc{Case 2.} Suppose $k$ is even. It is not difficult to see that, being $4$-regular, $G$ is $2$-edge connected. By Remark \ref{G'}, we can construct $G'$ so that $G'$ is a $4$-edge-connected $8$-regular graph. By Theorem \ref{regularfactors}, $G'$ has a $3$-factor, say $H'$. Let $G' \setminus H'$ be the $5$-factor of $G'$ obtained by removing the edges of $H'$ from $G'$.

\textsc{Sub-Case 2.1.} Let $k = 2d$, $d$ even. For each $c \in \mathbb{Z}_k^* \setminus \{\frac{1}{2}k, \frac{1}{4}k\}$, define $f_c:E(G') \rightarrow \mathbb{Z}_k^*$ by
$$f_c(e) =
\begin{cases}
2c & \text{if $e \in E(H')$ }    \\
k-c & \text{if $e \in E(G' \setminus H')$. }
\end{cases} $$
Observe that the sum of the labels of the edges incident to each of the vertices in $G'$ is equal to $5(k-c) + 3(2c) \equiv c \pmod{k}$. This shows that $f_c$ is a $c$-sum $k$-magic labeling of $G'$ for all $c \neq 0, \frac{1}{2}k, \frac{1}{4}k$. By Remark \ref{l'}, $\frac{1}{4}k \in \Sigma_k(G')$.

For each $c \in \mathbb{Z}_k^* \setminus \{\frac{1}{2}k, \frac{1}{4}k\}$, define $\ell_c: E(G) \rightarrow \mathbb{Z}_k^*$ by $\ell_c(e_i) = f_c(e_i) + f_c(e_i')$ for all $i$, $1 \leq i \leq m$. Clearly, $\ell_c$ is a $c$-sum $k$-magic labeling of $G$ for each $c \in \mathbb{Z}_k^* \setminus \{\frac{1}{2}k, \frac{1}{4}k\}$. By Remark \ref{l'}, we see that $\mathbb{Z}_k^* \setminus \{\frac{1}{2}k\} \subset \Sigma_k(G)$.

By Theorem \ref{Petersen}, $G$ is $2$-factorable. Let $G_1$ and $G_2$ be the two $2$-factors of $G$. Label the edges in $G_1$ with $d$ and the edges in $G_2$ with $\frac{1}{2}(k-d)$. This shows that $d = \frac{1}{2}k \in \Sigma_k(G)$.

\textsc{Sub-Case 2.2.} Let $k = 2d$, $d\geq3$ odd. Observe that, for $c \neq 0, \frac{1}{2}k$, the labeling $\ell_c$ in Sub-Case 2.1 is a $c$-sum $k$-magic labeling of $G$. To complete the proof, we only need to show that $\frac{1}{2}k \in \Sigma_k(G)$.

Let $d \neq 3$ and $9$. We give a labeling for the factors of $G'$ defined above (namely, $H'$ and $G' \setminus H'$) and the $2$-factors of $G$ (namely, $G_1$ and $G_2$) to show that $G$ is $d$-sum $k$-magic.

Let $f: E(G) \rightarrow \mathbb{Z}_k^*$ be defined by
$$f(e) =
\begin{cases}
d+1 & \text{if $e \in E(G_1)$}    \\
\frac{1}{2}(k-d-1) & \text{if $e \in E(G_2)$.}
\end{cases} $$
Clearly, $f$ is $(d+1)$-sum $k$-magic labeling of $G$.

Let $g':E(G') \rightarrow \mathbb{Z}_k^*$ be defined by
$$g'(e) =
\begin{cases}
k-2 & \text{if $e \in E(H')$ }    \\
1 & \text{if $e \in E(G' \setminus H')$. }
\end{cases} $$
Define also $g:E(G) \rightarrow \mathbb{Z}_k^*$ by $g(e_i) = g'(e_i) + g'(e_i')$ for all $i$, $1 \leq i \leq m$. Note that $g'$ is a $(k-1)$-sum $k$-magic labeling of $G'$, so $g$ is a $(k-1)$-sum $k$-magic labeling of $G$.

Finally, define $\ell :E(G) \rightarrow \mathbb{Z}_k^*$ by $\ell(e) = f(e) + g(e)$ for all $e \in E(G)$. Since $f$ and $g$ are $(d+1)$-sum and $(k-1)$-sum $k$-magic labeling of $G$ respectively, $\ell $ is a $d$-sum $k$-magic labeling of $G$.

Suppose $d = 3$ or $9$. Define $g':E(G') \rightarrow \mathbb{Z}_k^*$ be defined by
$$g'(e) =
\begin{cases}
2x & \text{if $e \in E(H')$ }    \\
1 & \text{if $e \in E(G' \setminus H')$, }
\end{cases} $$
where $x = 1$ if $d = 3$, and $x = 3$ if $d = 9$. Note that $g'$ is a $5$-sum $k$-magic labeling of $G'$. Define a labeling $g$ on $G$ by $g(e_i) = g'(e_i) + g'(e_i')+1$ for all $i$, $1 \leq i \leq m$. Note that $g$ is a $d$-sum $k$-magic labeling on $G$. Thus, $d = \frac{1}{2}k \in \Sigma_k(G)$, and so $G$ is completely $k$-magic.

Suppose $r \geq 3$ is odd. By Theorem \ref{Petersen}, $G$ is $r$-factorable. By Theorem \ref{odd-regular}, the $r$-factors of $G$ are completely $k$-magic for all $k \geq 5$. Thus, by Theorem \ref{factor2}, $G$ is also completely $k$-magic.
		
If $r \geq 4$ is even, then, by Theorem \ref{Petersen}, $G$ has a $6$-factor, say $H$. Using the case for $r$ is odd, $H$ is completely $k$-magic. Thus, by Theorem \ref{factor2}, $G$ is also completely $k$-magic.

(2(i)) By Theorem \ref{Petersen}, $G$ is $2$-factorable. Let $G_1, G_2, \dots, G_r$ be the $2$-factors of $G$. If $k$ is odd, then, by Remark \ref{rc_n}, $\mathbb{Z}_k^* \subseteq \sum_k(G_i)$ for all $i$, $1 \leq i \leq r$. For each $i$ and $c \in \mathbb{Z}_k^*$, let $\ell_c^i$ be a $c$-sum $k$-magic labeling of $G_i$. We consider two cases.
	
\textsc{Case 1.} Suppose $r \equiv 1 \pmod{k}$. For each $c \in \mathbb{Z}_k^*$, define $\ell_c:E(G) \to \mathbb{Z}_k^*$ by
$$\ell_c(e)=
\begin{cases}
\ell_c^1(e) & \text{if $e \in E(G_1)$} \\
\ell_1^i(e) & \text{if $e \in E(G_i)$ for some $i=2,3,\ldots,r$.}
\end{cases}$$
Note that $\ell_c$ is a $c$-sum $k$-magic labeling of $G$ for all $c\neq 0$.
	
\textsc{Case 2.} Suppose $r \not\equiv 1 \pmod{k}$. For each $c \in \mathbb{Z}_k^* \setminus \{r-1 \pmod{k}\}$, define $l_c:E(G) \rightarrow \mathbb{Z}_k^*$ by	
$$l_c(e)=
\begin{cases}
l_{c-x}^1(e) & \text{if $e \in E(G_1)$} \\
l_1^i(e) & \text{if $e \in E(G_i)$ for some $i=2,3,\ldots,r$,}
\end{cases}$$
where $x \equiv r-1 \pmod{k}$. The sum of the labels of the edges incident to each vertex is $c \pmod{k}$. Thus, $G$ is $c$-sum $k$-magic for each $c \neq x$. By Remark \ref{l'}, $G$ is $x$-sum $k$-magic since $G$ is $(k-x)$-sum $k$-magic. In this case, $G$ is completely $k$-magic.
	
(2(ii)) This follows from Remark \ref{rc_n}, Lemma \ref{odd order lemma}, and Theorem \ref{factor2}.
\end{proof}

\begin{theorem}\label{4-magic}
Let $r \geq 3$, and $G$ a zero-sum $4$-magic $r$-regular graph. Then
\begin{enumerate}
\item[{\rm (1)}] If the order of $G$ is even, then $G$ is completely $4$-magic.
\item[{\rm (2)}] If the order of $G$ is odd, then $\Sigma_4(G) = \{0, 2\}$.
\end{enumerate}
\end{theorem}

\begin{proof}
(1) Suppose the order of $G$ is even. We consider two cases.

\textsc{Case 1.} Suppose $r$ is odd. Clearly, $\gcd(r,4) = 1$, and so, by Theorem \ref{TheoremAfterLemmaFromDummit&Foote}, $\mathbb{Z}_4^* \subset \Sigma_4(G)$.

\textsc{Case 2.} Suppose $r = 2x$ for some $x \geq 2$. We consider two sub-cases.

\textsc{Sub-Case 2.1.} Suppose $x$ is odd. By Theorem \ref{Petersen}, $G$ is $x$-factorable. Let $G_1$ and $G_2$ be the two edge-disjoint $x$-factors of $G$. From Case 1, $\mathbb{Z}_4^*$ is a subset of both $\Sigma_4(G_1)$ and $\Sigma_4(G_2)$. Thus, we have $\mathbb{Z}_4^* \subset \Sigma_4(G)$.

\textsc{Sub-Case 2.2.} Suppose $x$ is even. If $x = 2$, then, as observed previously, $G$ is $2$-edge connected. By Remark \ref{G'}, we can construct $G'$ so that $G'$ is a $4$-edge-connected $8$-regular graph. By Theorem \ref{regularfactors}, $G'$ has a $5$-factor, say $H'$. Let $G' \setminus H'$ be the $3$-factor of $G'$ obtained by removing the edges of $H'$ from $G'$.

Define $f:E(G') \rightarrow \mathbb{Z}_k^*$ by
$$f(e) =
\begin{cases}
1 & \text{if $e \in E(H')$ }    \\
3 & \text{if $e \in E(G' \setminus H')$. }
\end{cases} $$
Observe that the sum of the labels of the edges incident to each of the vertices in $G'$ is $5(1) + 3(3) \equiv 2 \pmod{4}$. Define $\ell: E(G) \rightarrow \mathbb{Z}_k^*$ by $\ell(e_i) = \frac{1}{2}[f(e_i) + f(e_i')]$ for all $i$, $1 \leq i \leq m$. Clearly, $\ell$ is a $1$-sum $k$-magic labeling of $G$. By Remark \ref{l'}, we see that $G$ is $3$-sum $4$-magic as well.

To show that $G$ is $2$-sum $4$-magic, we consider a different labeling for $G$. By Theorem \ref{Petersen}, $G$ is $2$-factorable. Let $G_1$ and $G_2$ be the two $2$-factors of $G$. Label the edges in $G_1$ with $2$ and the edges in $G_2$ with $1$. This shows that $G$ is $2$-sum $4$-magic.

Suppose $x \ge 4$. By Theorem \ref{Petersen}, $G$ has a $2y$-factor for each $1 \leq y \leq x$. In particular, $G$ has a $6$-factor, say $H$. Let $G \setminus H$ be the $(2x-6)$-factor of $G$ obtained by removing the edges of $H$ from $G$. By Sub-Case 2.1, $\mathbb{Z}_4^*$ is a subset of both $\Sigma_4(H)$ and $\Sigma_4(G \setminus H)$. Again, it is not difficult to see that $\mathbb{Z}_4^* \subset \Sigma_4(G)$.

(2) Suppose the order of $G$ is odd. In this case, we only consider $2r$-regular graphs, $r \geq 2$. By Lemma \ref{odd order lemma}, $G$ is not $c$-sum $4$-magic for both $c = 1$ and $c=3$. To show that $G$ is $2$-sum $4$-magic, observe that, by Theorem \ref{Petersen}, $G$ is $2$-factorable. Let $G_1,G_2, \ldots, G_r$ be the $r$ edge-disjoint $2$-factors of $G$. Label the edges in $G_1$ with $1$, and label the edges in $G_i$ with $2$ for all $i \neq 1$. This labeling shows that $G$ is $2$-sum $4$-magic. By assumption, $0 \in \Sigma_4(G)$. Thus, $\Sigma_4(G) = \{0,2\}$.
\end{proof}

The last theorem to complete the proof of the Main Theorem characterizes all completely $3$-magic $r$-regular graphs, where $r \geq 3$.

\begin{theorem}
Let $G$ be an $r$-regular graph, where $r \geq 3$.
\begin{enumerate}
\item[{\rm (1)}] If $r \not\equiv 0 \pmod{3}$ or $r \equiv 0 \pmod{6}$, then $G$ is completely $3$-magic.
\item[{\rm (2)}] If $r \equiv 0 \pmod{3}$ and $r$ odd, then $G$ is completely $3$-magic if and only if $G$ has a factor $H$ such that $d_H(v) \equiv 1 \pmod{3}$ for all $v \in V(H)$.
\end{enumerate}
\end{theorem}

\begin{proof} (1) Suppose $r \equiv 1 \pmod{3}$. By Theorem \ref{Akbari&Rahmati&ZareTheorem13}, $G$ is zero-sum $3$-magic. By labeling the edges of $G$ with $1$, the sum of the labels of the edges incident to each vertex is $r \equiv 1 \pmod{3}$, and $G$ is $1$-sum $3$-magic. By Remark \ref{l'}, $G$ is also $2$-sum $3$-magic.

Suppose $r \equiv 2 \pmod{3}$. By Theorems \ref{Akbari&Rahmati&ZareTheorem13} and \ref{Dong&Wang}, $G$ is zero-sum $3$-magic. By using Remark \ref{l'} again and by labeling the edges of $G$ with $2$, it follows that $G$ is $1$-sum and $2$-sum $3$-magic.

Let $r = 2(3y)$, $y \geq 1$. By Theorem \ref{Petersen}, $G$ has a $2z$-factor for each $1 \leq z \leq 3y$. Let $H$ be a $4$-factor of $G$ and $G \setminus H$ be the $(6y-4)$-factor of $G$. As considered above, both $H$ and $G \setminus H$ are completely $3$-magic. Thus, by Theorem \ref{factor2}, $G$ is completely $3$-magic.

(2) Suppose $G$ has a factor $H$ such that $d_H(v) \equiv 1 \pmod{3}$ for all $v \in V(H)$. Denote by $G\setminus H$ the factor of $G$ obtained by removing the edges in $G$ that is an edge in $H$. Since $G$ is $3x$-regular, $x \geq 1$, we have $d_{G \setminus H}(v) \equiv 2 \pmod{3}$ for all $v \in V(G \setminus H)$. Label each edge in $H$ with $2$ and each edge in $G \setminus H$ with $1$. Note that the sum of the labels of the edges incident to each vertex in $G$ is $2(1) + 1(2) \equiv 1 \pmod{3}$. This shows that $G$ is $1$-sum $3$-magic.

Conversely, suppose $G$ is completely $3$-magic. It follows that $G$ is $1$-sum $3$-magic. Since $G$ is $3x$-regular (where $x \geq 1$), for any $1$-sum $3$-magic labeling of $G$ and for each vertex $v \in V(G)$, $v$ must be incident to $p$ edges (where $p \equiv 1 \pmod{3}$) with label $2$ and $q$ edges (where $q \equiv 2 \pmod{3}$) with label $1$. Let $H'$ be a subgraph of $G$ such that an edge $e \in E(H')$ if and only if the label of $e$ is $2$. Clearly, $H'$ is a factor of $G$ and that $d_{H'}(v) \equiv 1 \pmod{3}$ for all $v \in V(H')$.
\end{proof}

\begin{corollary}
Let $G$ be an $r$-regular graph, where $r \geq 3$. Then $G$ is completely $3$-magic if and only if one of the following conditions holds:
\begin{enumerate}
\item[{\rm (1)}] $r \not\equiv 0 \pmod{3}$,
\item[{\rm (2)}] $r \equiv 0 \pmod{6}$, or
\item[{\rm (3)}] $r \equiv 0 \pmod{3}$, $r$ odd, and $G$ has a factor $H$ such that $d_H(v) \equiv 1 \pmod{3}$ for all $v \in V(H)$.
\end{enumerate}
\end{corollary}


\begin{thebibliography}{99}
\bibitem{Akbari&Rahmati&Zare} S. Akbari, F. Rahmati, and S. Zare, Zero-sum magic labelings and null sets of regular graphs, \emph{The Electronic Journal of Combinatorics} \textbf{21}(2) (2014), \#P2.17.

\bibitem{Akiyama&Kano} J. Akiyama and M. Kano, \emph{Factors and Factorizations of Graphs}, Springer-Verlag, 2011.

\bibitem{Bondy&Murty} J.A. Bondy and U.S.R. Murty, \emph{Graph Theory}, Springer, 2008.

\bibitem{Choi&Georges&Mauro} J.-O. Choi, J.P. Georges, and D. Mauro, On zero-zum $\mathbb{Z}_k$-magic labelings of $3$-regular graphs, \emph{Graphs and Combinatorics} \textbf{29} (2013), 387-398.

\bibitem{Dong&Wang} G. Dong and N. Wang, A conjecture on zero-sum $3$-magic labeling of $5$-regular graphs, arXiv:1406.6870v1, 2014.

\bibitem{Eniego&Garces} A.A. Eniego and I.J.L. Garces, On completely $k$-magic regular graphs, \emph{Applied Mathematical Sciences (Ruse)} \textbf{103} (2015), 5139-5148.

\bibitem{Gallai} T. Gallai, On factorisation of graphs, \textit{Acta Mathematica Academiae Scientiarum Hungarica} \textbf{1}(1) (1950), 133-153.

\bibitem{Petersen} J. Petersen, Die Theorie der regularen graphs, \emph{Acta Mathematica} (15) (1891), 193-220.

\bibitem{Salehi1} E. Salehi, Integer-magic spectra of cycle-related graphs, \emph{Iranian Journal of Mathematical Sciences and Informatics} \textbf{2} (2006), 53-63.

\bibitem{Sedlacek} J. Sedlacek, On magic graphs, \emph{Mathematica Slovaca} \textbf{26} (1976), 329-335.

\bibitem{Shiu&Low2} W.C. Shiu and R.M. Low, Group-magic labelings of graphs with deleted edges, \emph{Australasian Journal of Combinatorics} \textbf{57} (2013), 3-19.

\bibitem{Wang&Hu} T.-M. Wang and S.-W. Hu, Constant sum flows in regular graphs, In \emph{Frontiers in Algorithmics and Algorithmic Aspects in Information and Management}, M. Attalah, X.-Y. Li, and B. Zhu (Editors), Springer Berlin Heidelberg (2011), 168-175.
\end{thebibliography}
\end{document}